\documentclass[12pt,twoside]{article}
\usepackage{amsfonts}
\usepackage{amsmath}
\usepackage{amssymb}
\usepackage{theorem}
\usepackage{tikz-cd}
\newcommand{\0}{\emptyset}

\newcommand{\M}{{\cal M}}
\newcommand{\Hh}{{\cal H}}
\newcommand{\A}{{\cal A}}
\newcommand{\B}{{\cal B}}

\newcommand{\G}{{\cal G}}
\newcommand{\K}{{\cal K}}
\newcommand{\R}{{\cal R}}
\newcommand{\Po}{{\cal P}}
\newcommand{\Def}{{\rm Def}}
\newcommand{\End}{{\rm End}}
\newcommand{\Ker}{{\rm Ker}}
\newcommand{\Img}{{\rm Im}}
\newcommand{\Aut}{{\rm Aut}}
\newcommand{\bi}{\begin{itemize}}
\newcommand{\ei}{\end{itemize}}
\newtheorem{theorem}{Theorem}[section]
\newtheorem{lemma}[theorem]{Lemma}
\newtheorem{corollary}[theorem]{Corollary}
\newtheorem{proposition}[theorem]{Proposition}
\newtheorem{definition}[theorem]{Definition}
\newtheorem{fact}[theorem]{Fact}
\newtheorem{remark}[theorem]{Remark}
\theorembodyfont{\normalfont}

\newenvironment{proof}{\addvspace{8pt plus 2pt minus
    2pt}\emph{Proof}}{ $\square$\par\addvspace{8pt plus 2pt minus
    2pt}}

\title{Weak heirs, coheirs and the Ellis semigroups}
\author{Adam Malinowski \and Ludomir Newelski\thanks{Research supported by the Narodowe
    Centrum Nauki grant no. 2018/31/B/ST1/00357}}
\date{\mbox{}}

\begin{document}
\maketitle
\begin{abstract}
Assume $G\prec H$ are groups and $\A\subseteq\Po(G),\ \B\subseteq\Po(H)$ are algebras of sets closed under left group translation. Under some additional assumptions we find algebraic connections between the Ellis [semi]groups of the $G$-flow $S(\A)$ and the $H$-flow $S(\B)$. We apply these results in the model theoretic context. Namely, assume $G$ is a group definable in a model $M$ and $M\prec^* N$. Using weak heirs and weak coheirs we point out some algebraic connections between the Ellis semigroups $S_{ext,G}(M)$ and $S_{ext,G}(N)$. Assuming every minimal left ideal in $S_{ext,G}(N)$ is a group we prove that the Ellis groups of $S_{ext,G}(M)$ are isomorphic to closed subgroups of the Ellis groups of $S_{ext,G}(N)$.
\end{abstract}
\section*{Introduction}

Assume $M$ is a model for a relational language $L$, $T=Th(M)$ and $G=G(M)$ is an infinite group $0$-definable in $M$. All models of $T$ we consider are elementary submodels of a monster model $\M$ of $T$. In this paper we are interested in topological dynamics of some $G$-flows of types over $M$. We assume the reader is familiar with the basic notions of topological dynamics \cite{ellis, aus, gla} and model theory \cite{hodges, lascar, pillay} as well as topological dynamics of definable groups \cite{ne1}, although we shall recall some of them.

Let $\Def_{ext,G}(M)$ be the algebra of externally definable subsets of $G$ and let $S_{ext,G}(M)$ be the Stone space of ultrafilters in $\Def_{ext,G}(M)$, called external $G$-types over $M$. $S_{ext,G}(M)$ is a $G$-flow, in model theory it is a counterpart of the universal $G$-flow $\beta G$ in topological dynamics. In particular the $G$-flow $S_{ext,G}(M)$ is isomorphic to its Ellis semigroup, hence it carries the induced semigroup operation $*$ that is continuous in the first coordinate (a semigroup like that is called left-continuous).

There is a question, how are the flows $S_{ext,G}(M)$ related for various models $M$ of $T$, in particular are the semigroups $(S_{ext,G}(M),*)$ and their Ellis groups related algebraically? \cite{ne2,ne3} contain some results on these questions. The results of \cite{ne1} and \cite[Proposition 3.1]{krupil} indicate that the quotient groups of $G(\M)$ by its model theoretic $M$-connected components are homomorphic images of the semigroup $S_{ext,G}(M)$ and also of its Ellis groups.

In order to compare the $G$-flow $S_{ext,G}(M)$ and the $G(N)$-flow $S_{ext,G}(N)$ for $M,N\models T$ we need to be able to compare the algebras $\Def_{ext,G}(M)$ and $\Def_{ext,G}(N)$. So it is natural to assume that $M\prec N$, but it is not enough. We need a stronger relation between $M$ and $N$, called $*$-elementary inclusion $\prec^*$, we recall it later.

Assume $M\prec^* N$. In \cite[Theorem 4.1]{ne2} we proved that if there is a generic type in $S_{ext,G}(N)$ (equivalently, $S_{ext,G}(N)$ has a unique minimal subflow \cite{ne1}), then the Ellis groups of $S_{ext,G}(M)$ are homomorphic images of subgroups of the Ellis groups of $S_{ext,G}(N)$. Here we prove that if every minimal left ideal in $S_{ext,G}(N)$ is a group, then the Ellis groups of $S_{ext,G}(M)$ are isomorphic to closed subgroups of the Ellis groups of $S_{ext,G}(N)$. The assumption that every minimal left ideal in $S_{ext,G}(N)$ is a group is dual to the assumption that $S_{ext,G}(N)$ contains a generic type.

We present our results in a general combinatorial set up, making the presentation easier. So we consider infinite groups $G\prec H$, a $G$-algebra $\A\subseteq\Po(\G)$ and an $H$-algebra $\B\subseteq\Po(H)$ and under some additional assumptions we find some algebraic connections between the Ellis [semi]groups of $S(\A)$ and $S(\B)$. The important tools in our proofs are the notions of weak heir from \cite{ne2} and weak coheir (introduced here). We think they are of independent interest, so we elaborate on them. We characterize them in the case where $T$ is stable, by means of local forking. We provide some examples.

\section{Preliminaries}
Assume $G$ is an infinite (discrete) group. We use $e$ to denote the identity element of $G$ and also of any other group. Assume that $\A$ is a  $G$-algebra, that is an algebra of subsets of $G$ closed under left translation by elements of $G$. In this case its Stone space $S(\A)$ is naturally a $G$-flow, with $G$ acting on $S(\A)$ by left translation. For $A\in\A$, $S(\A)\cap [A]$ denotes the set $\{p\in S(\A):A\in p\}$, a clopen subset of $S(\A)$. When $p\in S(\A)$ then $cl(Gp)$ is the subflow of $S(\A)$ generated by $p$.

For $g\in G$ let $l_g:S(\A)\to S(\A)$ be the left translation by $g$. Let $E(S(\A))$ be the topological closure of the set $\{l_g:g\in G\}$ in the space of functions $S(\A)\to S(\A)$ with the topology of pointwise convergence. $E(S(\A))$ is a $G$-flow, where $g\cdot f = l_g\circ f$ for $g\in G$ and $f\in E(S(\A))$. $E(S(\A))$ is also a compact left-continuous semigroup, with respect to composition of functions, called the Ellis semigroup of $S(\A)$.

Assume $S$ is a semigroup. We write $I\triangleleft S$ [$I\triangleleft_m S$] when $I$ is a [minimal] left ideal in $S$. We say that $u\in S$ is an idempotent if $u^2=u$. The next fact is a fundamental structural result of Ellis.

\begin{fact}[\mbox{\cite{ellis},\cite[Proposition 4.2]{eek}}]\label{1.1}
Assume $S$ is a compact Hausdorff left-continuous semigroup and let $J$ be the set of idempotents in $S$.\\
(1) Given $I\triangleleft_m S$, the set $J(I)=I\cap J$ is non-empty.\\
(2) For every $I\triangleleft_m S$ and $u\in J(I)$, the set $uI$ is a subgroup of $I$ and $I$ is a disjoint union of such subgroups.\\
(3) The groups $uI,I\triangleleft_m S,u\in J(I)$, are all isomorphic.\\
(4) If $a\in S$ and $I\triangleleft_m S$, then $Ia\triangleleft_m S$.
\end{fact}
The groups $uI$ appearing in Fact \ref{1.1} are called the Ellis subgroups of $S$. $E(S(\A))$ is a compact left-continuous semigroup, hence Fact \ref{1.1} applies. We call its Ellis subgroups the Ellis groups of the flow $S(\A)$.

When $\A=\Po(G)$, then $S(\A)=\beta G$ is naturally isomorphic to $E(\beta G)$. This fact has been generalized in \cite{ne4} to some other $G$-algebras as follows.

\begin{definition}[\mbox{\cite[Definition 2.2]{ne4}}]\label{1.2}
(1) For $p\in S(\A)$ we define a function $d_p:\A\to\Po(G)$ by
$$d_p(A)=\{g\in G: g^{-1}A\in p\}$$
Clearly $d_p:\A\to \Po(G)$ is a homomorphism of $G$-algebras.\\
(2) We say that $\A$ is $d$-closed if $\A$ is closed under $d_p$ for every $p\in S(\A)$, that is $d_p[\A]\subseteq\A$. Notice that in this case $d_p$ belongs to $\End(\A)$, the semigroup of endomorphisms of $G$-algebra $\A$.\\
(3) If $\A$ is $d$-closed, then let $d:S(\A)\to \End(\A)$ be the function mapping $p$ to $d_p$.
\end{definition}
So the algebra $\Po(G)$ is $d$-closed. By \cite[Lemma 1.2]{ne2} also the $G$-algebra $\Def_{ext,G}(M)$ is $d$-closed. The next fact lists  the basic properties of $d$-closed $G$-algebras.\footnote{For every $G$-algebra $\A$ there is a unique smallest $d$-closed $G$-algebra $\A^d$ containing $\A$, called the $d$-closure of $\A$. We have that $S(\A^d)$ is isomorphic to $E(S(\A))$ as a $G$-flow and semigroup. This will be proved in a forthcoming paper.}

\begin{fact}[\mbox{\cite[Proposition 2.4]{ne4}}]\label{1.3} Assume $\A$ is $d$-closed.\\
(1) The function $d:S(\A)\to \End(\A)$ is a bijection.\\
(2) $d$ induces on $S(\A)$ a semigroup operation $*$ so that $d:S(\A)\to\End(\A)$ becomes an isomorphism of semigroups and for $p,q\in S(\A)$ we have $d_{p*q}=d_p\circ d_q$. Also for $A\in\A$ we have
$$A\in p*q\iff d_q(A)\in p$$
(3) For $p\in S(\A)$ let $l_p:S(\A)\to S(\A)$ be the left translation by $p$, that is $l_p(q)=p*q$. Then $E(S(\A))=\{l_p:p\in S(\A)\}$ and the function $l:S(\A)\to E(S(\A))$ mapping $p$ to $l_p$ is an isomorphism of $G$-flows and of semigroups. In particular, $(S(\A),*)$ is a compact left-continuous semigroup.
\end{fact}

When $\A$ is $d$-closed and $p\in S(\A)$, we associate with $d_p\in \End(\A)$ its kernel $\Ker d_p$ and image $\Img d_p$. Clearly $\Ker d_p$ is a $G$-ideal in $\A$ and $\Img d_p$ is a $G$-subalgebra of $\A$.

\begin{fact}[\cite{ne2,ne4}]\label{1.4} Assume $\A$ is $d$-closed.\\
(1) Assume $p,q\in S(\A)$. Then $\Ker d_p\subseteq \Ker d_q$ iff $cl(Gp)\supseteq cl(Gq)$.\\
(2) If $S$ is a subgroup of $S(\A)$, then the endomorphisms $d_p,p\in S$ have common kernel $K(S)=\Ker d_p$ and image $R(S)=\Img d_p$. $R(S)$ is a section of the family of $K(S)$-cosets in $\A$, hence $R(S)\cong \A / K(S)$. Also, every $d_p,p\in S$ restricted to $R(S)$ is a $G$-algebra automorphism of $R(S)$ and the mapping $p\mapsto d_p|_{R(S)}$ is a group monomorphism $S\to \Aut(R(S))\cong\Aut(\A/ K(S))$.\\
(3) Let $\K = \{\Ker d_p: p\in I\triangleleft_m S(\A)\}$ and $\R = \{\Img d_p:p\in I\triangleleft_m S(\A)\}$. Then the mapping $uI\mapsto (K(uI), R(uI))$ is a bijection between the Ellis subgroups of $S(\A)$ and the product $\K\times\R$. Moreover, $K(uI)=K(u'I')$ iff $I=I'$ and for every $I\triangleleft_mS(\A)$ and $R\in\R$ there is a unique $u\in J(I)$ with $R=R(uI)$.
\end{fact}

We shall need a characterization of the algebras from $\R$. We say that $A\in \A$ is generic (in $G$) if some finitely many left translates of $A$ cover $G$. We say that a point $p\in S(\A)$ is generic if every set in $p$ is generic. We say that a $G$-subalgebra $R$ of $\A$ is generic if every non-empty element of $R$ is a generic subset of $G$. We say that $A\in\A$ is strongly generic if it generates a generic $G$-subalgebra of $\A$.

\begin{fact}[\cite{ne3,ne4}] \label{1.5} Assume $\A$ is $d$-closed. The $G$-algebras in $\R$ are exactly the maximal generic $G$-subalgebras of $\A$.
\end{fact}

{\bf The combinatorial set-up} From now on in this paper usually we assume $G$ is an infinite group (possibly with an additional first order structure) and $L_G$ is the language of $G$. Also we assume that $G\prec H$. $\Def(G)$ denotes the $G$-algebra of subsets of $G$ definable by formulas of $L_G$, using parameters from $G$.

Given $A\in\Def(G)$ definable by an $L_G$-formula $\varphi(x)$, let $A^{\#}$ be the subset of $H$ defined by $\varphi(x)$. Throughout, $\A$ denotes a $d$-closed $G$-algebra contained in $\Def(G)$ and $\B$ a $d$-closed $H$-algebra containing $A^{\#}$ for all $A\in\A$. We also assume that $\B|_G:=\{B\cap G:B\in \B\}$ equals $\A$.

Clearly, the function $\mbox{}^{\#}:\A\to\B$ mapping $A$ to $A^{\#}$ is a Boolean algebra monomorphism respecting left translation by elements of $G$. So we have a dual continuous surjection $r:S(\B)\to S(\A)$, called restriction. When $r(q)=p$, we say that $q$ extends $p$ and write $p\subseteq q$.

\begin{remark}\label{1.6} For every $q\in S(\B)$ and $A\in \A$ we have that $d_qA^{\#}\cap G=d_{r(q)}A$.
\end{remark}

{\bf The model theoretic set-up} This set-up appeared already in Introduction. So we have an $L$-structure $M$ and $T=Th(M)$. $G=G(M)$ is an infinite group $0$-definable in $M$ and we are interested in the $G$-flow $S_{ext,G}(M)$ of external $G$-types over $M$. We show how to reduce this set-up to the combinatorial one described above. We proceed as in \cite[Section 2]{ne2}.

For every $U\in \Def_{ext}(M)$ let $P_U$ be a new relational symbol and let $L_{ext,M}= L\cup\{P_U:U\in\Def_{ext}(M)\}$. Let $M_{ext}$ be the expansion of $M$ to an $L_{ext,M}$-structure, where $P_U$ is interpreted as $U$.

To compare topological dynamics of flows $S_{ext,G}(M)$ for various models $M$ of $T$ we need to consider not just elementary, but rather $*$-elementary extensions. We say that an $L$-structure $N$ is a $*$-elementary extension of $M$ (and write $M\prec^* N$) if $M\prec N$ and the new relational symbols of $L_{ext,M}$ are identified with some relational symbols of $L_{ext,N}$ so that $M_{ext}\prec N_{ext}|_{L_{ext,M}}$.

It is easy to construct $*$-elementary extensions of $M$. Namely, let $N^0$ be any elementary extension of $M_{ext}$ and let $N=N^0|_L$. For $U\in \Def_{ext}(M)$ let $U^N=P_U(N^0)$. By \cite[Lemma 2.1]{ne2}, $U^N\in \Def_{ext}(N)$, so the identification of $P_U$ with $P_{U^N }$ witnesses that $M\prec^* N$.

Assume $M\prec^* N$ and let $H=G(N)$. Then $G\prec H$, where $G$ and $H$ are considered with the structure induced from $M_{ext}$ and $N_{ext}|_{L_{ext,M}}$, respectively. Let $\A=\Def_{ext,G}(M)$ and $\B=\Def_{ext,G}(N)$. For $A\in\A$ let $A^{\#}= P_A(N)$. By \cite[Lemma 1.2]{ne2}, $G,H,\A$ and $\B$ satisfy the assumptions of combinatorial set-up. So the results on $S(\A)$ and $S(\B)$ in the combinatorial set-up will be valid also in the model theoretic set-up.

\section{Weak heirs and weak coheirs}
We work in the combinatorial set-up.
\begin{definition}\label{2.1} Assume $p\in S(\A)$ and $q\in S(\B)$.\\
(1) $q$ is a weak heir of $p$ if $d_qA^{\#} = (d_pA)^{\#}$ for every $A\in \A$.\\
(2) $q$ is a weak coheir of $p$ if for every $A\in\A$ and $s\in S(\B)$ we have that
$$d_s A^{\#}\in q\iff d_sA^{\#}\cap G\in p$$
\end{definition}
In (2) ``$\iff$" may be equivalently replaced by ``$\Rightarrow$" and also by Remark \ref{1.6}, $d_sA^{\#}\cap G=d_{r(s)}A$. In the model-theoretic set-up the definition of weak heir appeared already in \cite{ne2}.

\begin{lemma}\label{2.2} Assume $q$ is a weak heir or weak coheir of $p$. Then $p\subseteq q$.
\end{lemma}
\begin{proof}. Suppose that $p\not\subseteq q$. Hence for some $A\in p$ we have that $A^{\#}\not\in q$. Therefore $e\in d_pA$ and $e\not\in  d_qA^{\#}$.  So also $e\in (d_pA)^{\#}$ and $(d_pA)^{\#}\neq d_qA^{\#}$, showing that $q$ is not a weak heir of $p$.

Let $\hat{e}=\{B\in\B: e\in \B\}$. So $\hat{e}\in S(\B)$. We have that $d_{\hat{e}}A^{\#}=A^{\#}\not\in q$ and $d_{\hat{e}}A^{\#}\cap G=d_{r(\hat{e})}A=A\in p$, so $q$ is not a weak coheir of $p$.
\end{proof}

By Lemma \ref{2.2}, in Definition \ref{2.1} if $q$ is a weak heir or weak coheir of $p$, then $p=r(q)$. Therefore in this situation we say also that $q$ is a weak [co]heir over $\A$. Regarding Definition \ref{2.1}(2), $d_{r(s)}A\in r(q)$ is equivalent to $(d_{r(s)}A)^{\#}\in q$.

In the model theoretic set-up, when $p\in S_{ext,G}(M)$ and $q\in S_{ext,G}(N)$ is a weak [co]heir of $p$, we say that $q$ is a weak [co]heir over $N$.

The next lemma provides an alternative definition of weak heir and weak coheir.
\begin{lemma}\label{2.3} Assume $p\in S(\A)$ and $q\in S(\B)$.\\
(1) $q$ is a weak heir of $p$ iff for every $A,B\in \A$ and $h\in H$, if $h\in B^{\#}$ and $h^{-1}A^{\#}\in q$, then for some $g\in G$, $g\in B$ and $g^{-1}A\in p$.\\
(2) $q$ is a weak coheir of $p$ iff $p=r(q)$ and for every $A,B\in\A$ and $s\in S(\B)$, if $d_sA^{\#}\cap B^{\#}\in q$, then $d_sA^{\#}\cap B^{\#}\cap G\neq\0$.
\end{lemma}
\begin{proof}. (1)$\Rightarrow$ Assume $q$ is a weak heir of $p$. By Lemma \ref{2.2}, $p=r(q)$. Take any $A,B\in\A$ and $h\in H$ satisfying  $h\in B^{\#}$ and $h^{-1}A^{\#}\in q$. Then
$$h\in d_qA^{\#}\cap B^{\#}=(d_{r(q)}A\cap B)^{\#}$$
Since $G\prec H$, there is $g\in G$ with $g\in d_{r(q)}A\cap B$. It follows that $g\in B$ and $g^{-1}A\in r(q)$, as required.

$\Leftarrow$ Suppose for contradiction that the right hand side condition holds and $q$ is not a weak heir of $p$. Hence $(d_pA)^{\#}\neq d_qA^{\#}$ for some $A\in\A$. Replacing $A$ with $G\setminus A$ if necessary, we can assume that there is $h\in d_qA^{\#}\setminus (d_pA)^{\#}$.

Let $B=G\setminus d_pA$. Then $h\in B^{\#}$ and $h^{-1}A^{\#}\in q$, so  by our assumptions there is $g\in G$ satisfying $g\in B$ and $g^{-1}A\in p$. It follows that $g\in B$ and $g\in d_pA$, a contradiction.

(2)$\Rightarrow$ Assume $q$ is a weak coheir of $p$. By Lemma \ref{2.2}, $p=r(q)$. Take any $A,B\in\A$ and $s\in S(\B)$ with $d_sA^{\#}\cap B^{\#}\in q$. By the definition of weak coheir we get that
$$(d_{r(s)}A\cap B)\in p,$$
hence $d_{r(s)}A\cap B\neq\0$. So the right hand side condition holds.

$\Leftarrow$ Suppose for contradiction that the right hand side condition holds and $q$ is not a weak coheir of $p=r(q)$. hence for some $A\in\A$ and $s\in S(\B)$ we have
$$(d_{r(s)}A)^{\#}\triangle d_sA^{\#}\in q$$
Replacing $A$ with $G\setminus A$ if necessary, we can assume that $d_sA^{\#}\setminus(d_{r(s)}A)^{\#}\in q$. Let $B=G\setminus d_{r(s)}A$. Then $d_sA^{\#}\cap B^{\#}\in q$, but $d_sA^{\#}\cap B^{\#}\cap G=d_{r(s)}A\cap B=\0$, a contradiction.
\end{proof}

We define coheirs in the standard way.
\begin{definition}\label{2.4} (1) Let $q\in S(\B)$. We say that $q$ is a coheir over $\A$ if $B\cap G\neq\0$ for every $B\in q$.\\
(2) For $p\in S(\A)$ let
$$p^{\B}=\{B\in \B: B\cap G\in p\}$$
So $p^{\B}\in S(\B)$ is a coheir over $\A$ extending $p$.
\end{definition}
\begin{remark}\label{2.5} Assume $q\in S(\B)$. Then $q$ is a coheir over $\A$ iff $q=p^{\B}$, where $p=r(q)$. In particular every ultrafilter $p\in S(\A)$ has a unique coheir extension $p^{\B}$ in $S(\B)$ and the function $p\mapsto p^{\B}$ is a continuous embedding $S(\A)\to S(\B)$.
\end{remark}
So if $q\in S(\B)$ is a coheir over $\A$, we say also that $q$ is a coheir extension of $p=r(q)$.

We can identify formulas $\varphi(x)$ of $L_G$ with parameters from $H$ with the subsets of $H$ they define. For $A,B\in\A$ let $\varphi_{A,B}(x,y)$ be a formula of $L_G$ (with parameters from $G$) saying that ``$x\in y^{-1}A\land y\in B$". The condition in Lemma \ref{2.3}(1)  characterizing weak heirs says that if there is $h\in H$ with $\varphi_{A,B}(x,h)\in q$, then there is $g\in G$ with $\varphi_{A,B}(x,g)\in q$. This shows that the notion of weak heir is indeed a weakening of the notion of heir. Similarly by Lemmas \ref{2.5} and \ref{2.3}(2) we see that every coheir is a weak coheir. Consequently we get the following lemma.

\begin{lemma}\label{2.6} Every $p\in S(\A)$ extends to a weak heir and to a weak coheir in $S(\B)$.
\end{lemma}
 \begin{proof}. Let $q_0=p^{\B}$. By Remark \ref{2.5} and Lemma \ref{2.3}(2), $q_0$ is a coheir of $p$, hence a weak coheir of $p$. Regarding the weak heir case, first extend $p$ to a complete type $p'\in S(G)$ and let $q_1'\in S(H)$ be a heir extension of $p'$. We can regard $q_1'$ as an ultrafilter in the algebra $\Def(H)$. Extend $q_1'$ to $q_1''\in \beta H$ and let $q_1=q_1''\cap \B$. By the discussion before Lemma \ref{2.6},  $q_1\in S(\B)$ is a weak heir of $p$.
 \end{proof}

Let $CH(\B/\A)=\{q\in S(\B): q \mbox{ is a coheir over }\A\}$ and similarly define the sets $WCH(\B/\A)$ and $WH(\B/\A)$ of weak coheirs and weak heirs respectively.

\begin{remark}\label{2.7} $CH(\B/\A)$ and $WCH(\B/\A)$ are closed subsets of $S(\B)$ and $CH(\B/\A)\subseteq WCH(\B/\A)$. The sets $CH(\B/\A), WCH(\B/\A)$ and $WH(\B/\A)$ are non-empty.
\end{remark}
Even though the set $WH(\B/\A)$ need not be closed (see an example at the end of this paper), it is a union of certain closed sets. For $p\in S(\A)$ let $WH_p(\B)=\{q\in WH(\B/\A): p\subseteq q\}$.
\begin{lemma}\label{2.8} The set $WH_p(\B)$ is closed and non-empty. $WH(\B/\A)=\bigcup_{p\in S(\A)}WH_p(\B)$.
\end{lemma}
\begin{proof}. $WH_p(\B)\neq\0$ by Lemma \ref{2.6}. Let
$$p_{\B}=\{h^{-1}A^{\#}: h\in (d_pA)^{\#}\mbox{ and } A\in\A\}$$
By Definition \ref{2.1}(1), $WH_p(\B)=\{q\in S(\B): p_{\B}\subseteq q\}$, so it is closed. The last part of the lemma is obvious.
\end{proof}

We see that our definitions of weak heir and weak coheir consist in restricting the definitions of heir and coheir to some special formulas.
The next proposition justifies this particular restriction.
\begin{proposition}\label{2.9} Assume $q\in S(\B)$.\\
(1) $q$ is a weak heir over $\A$ iff $r(s*q)=r(s)*r(q)$ for every $s\in S(\B)$.\\
(2) $q$ is a weak coheir over $\A$ iff $r(q*s)=r(q)*r(s)$ for every $s\in S(\B)$.
\end{proposition}
\begin{proof}. Note that for any $A\in\A$ and $s\in S(\B)$,
$$A\in r(s*q)\iff d_qA^{\#}\in s$$
$$A\in r(s)*r(q)\iff (d_{r(q)}A)^{\#}\in s$$

(1) If $q$ is a weak heir over $\A$, then the right hand sides of the above equivalences are equivalent for each $A,s$, so $r(s*q)=r(s)*r(q)$. On the other hand, if $q$ is not a weak heir over $\A$, then $d_qA^{\#}\neq(d_{r(q)}A)^{\#}$ for some $A\in \A$. Thus we can find $s\in S(\B)$ with $d_qA^{\#}\triangle(d_{r(q)}A)^{\#}\in s$ and then $r(s*q)\neq r(s)*r(q)$.

(2) follows directly from definition.
\end{proof}
In fact, Proposition \ref{2.9}(1)$\Rightarrow$ was already proved in \cite{ne2} in the model theoretic set-up.

For $h\in H$ let $\hat{h}=\{B\in \B: h\in H\}$ and $\hat{H}=\{\hat{h}:h\in H\}$. Clearly $\hat{H}$ is a dense subset of $S(\B)$. It is easy to see that the function $f:H\to\hat{H}$ mapping $h$ to $\hat{h}$ is a $*$-homomorphism, so $\hat{H}$ is a dense subgroup of $S(\B)$ and $f$ is a group epimorphism. Since $*$ is a continuous in the first coordinate we get the following corollary.

\begin{corollary}\label{2.10} Assume $q\in S(\B)$. Then $q$ is a weak heir over $\A$ iff $r(\hat{h}*q)=r(\hat{h}) * r(q)$ for every $h\in H$.
\end{corollary}
Notice also that $\hat{h}*q=hq$.
\begin{corollary}\label{2.11} The sets $CH(\B/\A)\subseteq WCH(\B/\A)$ and $WH(\B/\A)$ are sub-semigroups of $S(\B)$ and restrictions $r:WCH(\B/\A)\to S(\A)$ and $r:WH(\B/\A)\to S(\A)$ are $*$-epimorphisms. The restriction $r:CH(\B/\A)\to S(\A)$ is a semigroup isomorphism and homeomorphism. Its inverse is the function $j:S(\A)\to CH(\B/\A)$ mapping $p$ to $p^{\B}$.
\end{corollary}
\begin{proof}. We prove that the set $CH(\B/\A)$ is closed under $*$, the other claims of the corollary follow directly from earlier results. First, for every $q\in CH(\B/\A)$ and $A\in\B$ we have that
$$A\in q\iff A\cap G\in r(q)\mbox{ and } d_qA\cap G = d_{r(q)}(A\cap G)$$
Therefore for $p,q\in CH(\B/\A)$ and $A\in\B$ we have that
$$A\in p*q\iff d_qA\in p\iff d_qA\cap G\in r(p)\iff $$
$$d_{r(q)}(A\cap G)\in r(p)\iff A\cap G\in r(p)*r(q)$$
Hence $A\in p*q$ implies $A\cap G\neq\0$ and $p*q\in CH(\B/\A)$.
\end{proof}

Even though the left-continuous semigroup $WH(\B/\A)$ may not be closed, still it has the properties from Fact \ref{1.1} (see Proposition \ref{3.0.5}).
In the model theoretic set-up Corollary \ref{2.11} was already partially noticed in earlier papers. Namely in \cite{ne1} it was mentioned that the semigroup $S(\A)$ is isomorphic to $CH(\B/\A)$ (suitably rephrased in the model theoretic context) and in \cite{ne2} it was proved that $WH(\B/\A)$ is a semigroup and $r:WH(\B/\A)\to S(\A)$ is an epimorphism. We shall use Corollary \ref{2.11} later to relate algebraically the Ellis groups of $S(\A)$ and $S(\B)$. The semigroup $WH(\B/\A)\cap CH(\B/A)$ contains an important subgroup.
\begin{lemma}\label{2.12} Let $\hat{G}=\{\hat{h}\in\hat{H}: h\in G\}$.\\
(1) $\hat{G}$ is a subgroup of $S(\B)$ and the mapping $g\mapsto \hat{g}$ is a group epimorphism $G\to \hat{G}$.\\
(2) $\hat{G}\subseteq WH(\B/\A)\cap CH(\B/A)$, $\hat{e}$ is the unique identity element in $S(\B)$.\\
(3) For every $q\in S(\B)$ and $g\in G$ we have that
$$q\in WH(\B/\A)\iff \hat{g}*q\in WH(\B/\A)\iff q*\hat{g}\in WH(\B/A)$$
(4) Like (3), but with $WH$ replaced by $WCH$ [or $CH$].
\end{lemma}
\begin{proof}. The proof is by revealing the definitions and applying Proposition \ref{2.9} in (3) and (4).
\end{proof}

By Fact \ref{1.3}(3), elements of $S(\A)$ and $S(\B)$ may be regarded as functions from the Ellis semigroups $E(S(\A))$ and $E(S(\B))$. In Proposition \ref{2.11.1} we give yet another characterization of weak [co]heirs in terms of functions from $E(S(\A))$ and $E(S(\B))$.

Let $\A^{\#}$ be the $H$-subalgebra of $\B$ generated by the set $\{ A^{\#} : A \in \A \}$. So $\A^{\#}$ need not be $d$-closed. Each $q_0 \in S(\A)$ has a unique weak heir in $S(\A^{\#})$, that is, a type $q \in S(\A^{\#})$ such that $d_q A^{\#} = (d_{q_0} A)^{\#}$ for each $A \in \A$. We write $(q_0)^{\#} := q$.

Consider any $p_0 \in S(\A),p \in S(\B)$ and let $f_0 \in E(S(\A))$, $f \in E(S(\B))$ correspond to $p_0$ and $p$ via the standard isomorphisms from Fact \ref{1.3}(3). $p$ determines a function $f^{\#}:S(\A^{\#})\to S(\A^{\#})$ defined by
$f^{\#}(q) = p * q$, where $p*q\in S(\A^{\#})$ is defined by the formula from Fact \ref{1.3}(2). Since $p*q=\lim_{h\to p}hq$, we have that $f^{\#}\in E(S(\A^{\#}))$ and $f^{\#}$ is the unique function from $E(S(\A^{\#}))$ such that the following diagram commutes:

\begin{center}
\begin{tikzcd}
S(\B) \arrow{r}{f} \arrow{d}{r} & S(\B) \arrow{d}{r} \\
S(\A^{\#}) \arrow{r}{f^{\#}} & S(\A^{\#})
\end{tikzcd}
\end{center}
Also using the formula for $*$ from Fact \ref{1.3}(2) we get that $d_{p*q}=d_p\circ d_q$ for $p\in S(\B)$ and $q\in S(\A^{\#})$ .
\begin{proposition}\label{2.11.1}
Assume $p_0\in S(\A),p\in S(\B)$ and $f_0,f,f^{\#}$ are defined as above.\\
(1) $p$ extends $p_0$ iff the following diagram commutes:
\begin{center}\begin{tikzcd}
S(\A^{\#}) \arrow{r}{f^{\#}} & S(\A^{\#}) \arrow{d}{r} \\
S(\A) \arrow{u}{\#} \arrow{r}{f_0} & S(\A)
\end{tikzcd}
\end{center}
(2) $p$ is a weak heir of $p_0$ iff the following diagram commutes:
\begin{center}\begin{tikzcd}
S(\A^{\#}) \arrow{r}{f^{\#}} & S(\A^{\#}) \\
S(\A) \arrow{u}{\#} \arrow{r}{f_0} & S(\A) \arrow{u}{\#}
\end{tikzcd}
\end{center}
(3) $p$ is a weak coheir of $p_0$ iff the following diagram commutes:
\begin{center}\begin{tikzcd}
S(\A^{\#}) \arrow{d}{r} \arrow{r}{f^{\#}} & S(\A^{\#}) \arrow{d}{r} \\
S(\A) \arrow{r}{f_0} & S(\A)
\end{tikzcd}
\end{center}

\end{proposition}

\begin{proof}.
(1) We have that
\begin{align*}
p \text{ extends } p_0 & \iff (\forall q \in S(\A))(\forall A \in \A) \, \big( d_q A \in p_0 \Leftrightarrow (d_q A)^{\#} \in p \big) \\[1ex]
& \iff (\forall q \in S(\A))(\forall A \in \A) \, \big( d_q A \in p_0 \Leftrightarrow d_{q^{\#}} A^{\#} \in p \big) \\[1ex]
& \iff (\forall q \in S(\A))(\forall A \in \A) \, \big( A \in p_0 \ast q \Leftrightarrow A^{\#} \in p \ast q^{\#} \big) \\[1ex]
& \iff (\forall q \in S(\A))(\forall A \in \A) \, \big( A \in f_0(q) \Leftrightarrow A^{\#} \in f^{\#}(q^{\#}) \big) \\[1ex]
& \iff (\forall q \in S(\A))(\forall A \in \A) \, \big( A \in f_0(q) \Leftrightarrow A \in r(f^{\#}(q^{\#})) \big)
\end{align*}
and the last condition means that the corresponding diagram commutes.

(2) We have that
\begin{align*}
p \text{ is a weak heir of } p_0 & \iff (\forall B \in \A) \, (d_{p_0} B)^{\#} = d_p B^{\#} \\[1ex]
& \iff (\forall q \in S(\A))(\forall A \in \A) \, (d_{p_0} d_q A)^{\#} = d_p (d_q A)^{\#} \\[1ex]
& \iff (\forall q \in S(\A))(\forall A \in \A) \, (d_{p_0} d_q A)^{\#} = d_p d_{q^{\#}} A^{\#} \\[1ex]
& \iff (\forall q \in S(\A))(\forall A \in \A) \, (d_{p_0 * q} A)^{\#} = d_{p * q^{\#}} A^{\#} \\[1ex]
& \iff (\forall q \in S(\A))(\forall A \in \A) \, d_{(p_0 * q)^{\#}} A^{\#} = d_{p * q^{\#}} A^{\#} \\[1ex]
& \iff (\forall q \in S(\A))(\forall A \in \A) \, d_{f_0(q)^{\#}} A^{\#} = d_{f^{\#}(q^{\#})} A^{\#}.
\end{align*}
The last condition holds exactly when $f_0(q)^{\#} = f^{\#}(q^{\#})$ for each $q \in S(\A)$, which is equivalent to the commutativity of the corresponding diagram.

(3) By Proposition \ref{2.9} we have that
\begin{align*}
p \text{ is a weak coheir of } p_0 & \iff (\forall q \in S(\B)) \, r(p * q) = p_0 * r(q) \\[1ex]
& \iff (\forall q \in S(\B)) \, r(f(q)) = f_0(r(q)),
\end{align*}
where $r:S(\B)\to S(\A)$ is restriction.
The last condition is equivalent to the commutativity of the outer rectangle in the following diagram:
\begin{center}
\begin{tikzcd}
S(\B) \arrow{d}{r} \arrow{r}{f} & S(\B) \arrow{d}{r} \\
S(\A^{\#}) \arrow{d}{r} \arrow{r}{f^{\#}} & S(\A^{\#}) \arrow{d}{r} \\
S(\A) \arrow{r}{f_0} & S(\A)
\end{tikzcd}
\end{center}
Since the upper rectangle is commutative and $r : S(\B) \to S(\A^{\#})$ is surjective, the condition is equivalent to the commutativity of the lower rectangle.
\end{proof}

In order to be able  to compare the flows $S(\A)$ and $S(\B)$ we need a way to lift minimal left ideals in $S(\A)$ to minimal left ideals in $S(\B)$. This uses weak heirs and was already done in \cite[Lemma 2.4]{ne2} in the model theoretic set-up. Here we easily adapt the argument from \cite{ne2}. Recall from topological dynamics that a point $p$ in a $G$-flow $X$ is called almost periodic if $p$ belongs to a minimal subflow of $X$, equivalently the subflow $cl(Gp)\subseteq X$ generated by $p$ is minimal. In our situation the minimal flows in $S(\A)$ and $S(\B)$ are exactly the minimal left ideals. Their elements are exactly the almost periodic points. The next lemma is the main place where we use the assumption that $G\prec H$ in the combinatorial set-up and $M\prec^* N$ in the model theoretic set-up.
\begin{lemma}\label{2.13} (1) Assume $q\in S(\B)$ is a weak heir of $p\in S(\A)$. Then $r[cl(Hq)]=cl(Gp)$. If moreover $p$ is almost periodic, then $r[cl(Hq')]=cl(Gp)$ for every $q'\in cl(Hq)$.\\
(2) Assume $I\triangleleft_mS(\A)$. Then there is $I'\triangleleft_mS(\B)$ with $r[I']=I$. In particular, every almost periodic $p\in S(\A)$ extends to an almost periodic $q\in S(\B)$.
\end{lemma}
\begin{proof}. (1) For $h\in H$ we have that $hq=\hat{h}*q$. Since $*$ is left-continuous and the set $\{\hat{h}:h\in H\}$ is dense in $S(\B)$, we have that $cl(Hq)=S(\B)*q$. Proposition \ref{2.9}(1) implies that
$$r[cl(Hq)]= r[S(\B)*q]=S(\A)*p=cl[Gp]$$

In the case where $p$ is almost periodic, $cl(Gp)$ is a minimal left ideal in $S(\A)$. Since $r[cl(Hq')]$ is a $G$-invariant closed subset of $cl(Gp)$, we get that $r[cl(Hq)]=cl(Gp)$.

(2) Let $p\in I$ and let $q_0\in S(\B)$ be a weak heir of $p$. Let $I'\triangleleft_mS(B)$ be a minimal subflow of $cl(Hq_0)$. By (1), $r[I']=I$.
\end{proof}

\section{Transfer between $S(\A)$ and $S(\B)$: the Ellis groups}
In this section we shall relate algebraically the Ellis groups of $S(\A)$ and $S(\B)$, under some additional assumptions. Let $I\triangleleft_mS(\A)$ and $u\in J(I)$. First we lift  the group $uI$ to a group in $S(\B)$, possibly not an Ellis group. We do it in two ways, using the coheir and weak heir extensions.

First let $\G=j[uI]$, where $j:S(\A)\to CH(\B/\A)$ is the $*$-isomorphism from Corollary \ref{2.11}. So $\G$ is a subgroup of $S(\B)$ such that $r:\G\to uI$ is a group isomorphism. The group $\G$ is determined by the choice of $u\in J(I)$, so we may write $\G=\G_u$. The groups $\G_u,u\in J(I),I\triangleleft_mS(\A)$ are all isomorphic and their choice is canonical.

Next, we find a subgroup $\Hh$ of $WH(\B/\A)$ with the property that $r:\Hh\to uI$ is a group epimorphism. In the model theoretic set-up this has been done in \cite[Section 3]{ne2}. The construction from \cite{ne2} is easily adapted to our combinatorial set-up. We describe it briefly here, referring the reader to \cite[Section 3]{ne2} for details.

By Lemma \ref{2.8} and Corollary \ref{2.11}, the set $WH_u(\B)$ is a closed subset of $S(\A)$ and a sub-semigroup. Fix an $I'\triangleleft_mWH_u(\B)$. By \cite[Lemma 3.3(2)]{ne2} there is a common kernel $K'\subseteq\B$ of all $d_{q'},q'\in I'$. Also let $\R'_u=\{\Img d_{q'}:q'\in I'\}$. By Fact \ref{1.1}, $I'$ is a disjoint union of isomorphic groups $u'I',u'\in J(I')$. By \cite[Lemma 3.3]{ne2} there is $I^+\triangleleft S(\B)$ with $I'=WH_u(\B)\cap I^+$ (in fact, $I^+=cl(Hs)$ for any $s\in I'$). Notice also that by Fact \ref{2.13}, $r[I^+]=I$.

By \cite[Lemma 3.6]{ne2}, for every $q\in I$ the set $I_q':= WH_q(\B)\cap I^+$ is nonempty and for every $q'\in I_q'$ we have that $\Ker d_{q'}=K'$. Also for every $q\in uI$ we have that $\R'_u=\R'_q$, where $\R'_q:=\{\Img d_{q'}: q'\in I_q'\}$. Now fix an idempotent $u'\in J(I')$ and let $R'=\Img d_{u'}$. Let
$$\Hh=\{q'\in\bigcup_{q\in uI}I_q': \Img d_{q'}=R'\}$$
Clearly, $\Hh\subseteq WH(\B/\A)\cap I^+$. By \cite[Proposition 3.7]{ne2}, $\Hh$ is a subgroup of $S(\B)$ and $r:\Hh\to uI$ is a group epimorphism, with kernel $u'I'$. The group $\Hh$ is determined by the choice of $u'$, hence we may write $\Hh$ as $\Hh_{u'}$.

Using $\Hh$, in \cite{ne2} we proved in the model-theoretic set-up that the Ellis  groups of $S_{ext,G}(M)$ are homomorphic images of some subgroups of the Ellis groups of $S_{ext,G}(N)$, under the assumption that there are generic types in $S_{ext,G}(N)$ (equivalently, in $S_{ext,G}(N)$ there is exactly one minimal left ideal \cite{ne1}).  We will recall this result later in the combinatorial set-up, with an easier proof. Notice that the choice of the group $\Hh=\Hh_{u'}$ is canonical. We mentioned it in \cite{ne2}, here we make this claim precise in Proposition \ref{3.0.5}, that is a structural result on the semigroup $WH(\B/\A)$ parallel to Fact \ref{1.1}. Let $I_0^+=WH(\B/\A)\cap I^+$. Hence $I_0^+\triangleleft WH(\B/\A)$.

\begin{lemma}\label{3.0.1} (1) $r[I^+]=r[I_0^+]=I$.\\
(2) For every $q'\in I^+_0,\ I^+=cl(Hq')$ and $\Ker d_{q'}=K'$.\\
(3) If $q\in I$, then $I_q'\neq\0$ and if $q\in J(I)$, then $I_q'\triangleleft_m WH_q(\B)$.\\
(4) For $q\in I$ let $\R_q'=\{\Img d_{q'}:q'\in I_q'\}$. If $u_0\in J(I)$ and $q\in u_0I$, then $\R_q'=\R_{u_0}'$.\\
(5) For $v_0\in J(I_0^+)$ let
$$\Hh_{v_0}=\{q'\in\bigcup_{q\in u_0I}I_q': \Img d_{q'}=\Img d_{v_0}\}$$
where $u_0=r(v_0)$. Then $u_0\in J(I),\ \Hh_{v_0}\subseteq I_0^+$ is a group and $I_0^+$ is a disjoint union of the groups $\Hh_{v_0},v_0\in J(I_0^+)$.\\
(6) If $v_0\neq v_1\in J(I_0^+)$, then $\Img d_{v_0}\neq \Img d_{v_1}$.\\
(7) The groups $\Hh_v,v\in J(I^+_0)$ are isomorphic and $\Hh_v=vI_0^+$.\\
(8) $I_0^+\triangleleft_m WH(\B/\A)$.
\end{lemma}
\begin{proof}. (2)  and (3) are \cite[Lemma 3.6(1),(2)]{ne2}. Together with Lemma \ref{2.13} they imply (1).

(4) is \cite[Lemma 3.6(3)]{ne2} when $u_0=u$. However every $u_0\in J(I)$ can play the role of $u$ in generating $I^+$. Namely, by (3) $I_{u_0}'\triangleleft _m WH_{u_0}(\B)$ and by (2), $I^+=cl(Hu_0')$ for any $u_0'\in J(I_{u_0}')$. So the roles of $u$ and $u_0$ in $I^+$ are symmetrical hence \cite[Lemma 3.6(3)]{ne2} holds for arbitrary $u_0\in J(I)$. So (4) follows.

(5) Let $v_0\in J(I_0^+)$ and $u_0=r(v_0)$. We have that $u_0\in J(I)$ because $r:I_0^+\to I$ is a $*$-homomorphism. That $\Hh_{v_0}$ is a group with $r:\Hh_{v_0}\to u_0I$ being a group epimorphism follows from \cite[Proposition 3.7]{ne2} and the symmetrical role of $u_0$ and $u$ in $I^+$.

To see that the groups $\Hh_{v_0},v_0\in J(I_0^+)$ are disjoint, consider $v_0\neq v_1\in J(I_0^+)$. Let $u_i=r(v_i)\in J(I),\ i=0,1$.

If $u_0\neq u_1$, then $r[\Hh_{v_0}]=u_0I$, $r[\Hh_{v_1}]=u_1I$ and $u_0I\cap u_1I=\0$, hence also $\Hh_{v_0}\cap \Hh_{v_1}=\0$.

If $u_0=u_1$, then $v_0,v_1\in J(I_{u_0}')$ and $I_{u_0}'\triangleleft_m WH_{u_0}(\B)$, so $v_0\neq v_1$ implies $\Img d_{v_0}\neq\Img d_{v_1}$. But $\Img d_{v_i}=\Img d_x$ for every $x\in \Hh_{v_i}$, hence $\Hh_{v_0}$ and $\Hh_{v_1}$ are disjoint also in this case.

By (4) every $q'\in I_0^+$ belongs to some $\Hh_{v_0},v_0\in J(I_0^+)$, so we are done.

(6) Assume $v_0\neq v_1\in J(I_0^+)$. By (5) the groups $\Hh_{v_0}$ and $\Hh_{v_1}$ are disjoint and $\Ker d_{v_0}=\Ker d_{v_1}=K'$, hence by Fact \ref{1.4}, $\Img d_{v_0}\neq\Img d_{v_1}$.

(7) Let $v_0\neq v_1\in J(I_0^+)$. Since $K(\Hh_{v_0})=K(\Hh_{v_1})=K'$, the groups $\Hh_{v_i}$ may be regarded as subgroups of $\Aut(\B/K')$ (via the functiom $d$, since $R(\Hh_{v_i})$ is a section of the family of $K'$-cosets in  $\B$). Hence the function mapping $x$ to $v_1x$ is a group isomorphism $\Hh_{v_0}\to \Hh_{v_1}$. We see also that $\Hh_{v_0}=v_0I_0^+$.

(8) Let $x,y\in I_0^+$. Choose $v_0,v_1\in J(I_0^+)$ such that $x\in \Hh_{v_0}$ and $y\in \Hh_{v_1}$. It is enough to show that $y=zx$ for some $z\in I_0^+$.

When $v_0=v_1$, then $x,y\in \Hh_{v_0}$ and $z=yx^{-1}$ calculated in the group $\Hh_{v_0}$ will do. Otherwise, $v_1x\in \Hh_{v_1}$ and $z=z_0v_1$ will do, where $z_0=y(v_1x)^{-1}$ calculated in the group $\Hh_{v_1}$.
\end{proof}
\begin{lemma}\label{3.0.2} Assume $I^*\triangleleft_m WH(\B/\A)$. Then $I^*$ is of the form $I_0^+$ for some $I\triangleleft_m S(\A),\ u\in J(I),\ I'\triangleleft_m WH_u(\B)$ and $u'\in J(I')$.
\end{lemma}
\begin{proof}. Let $I=r[I^*]$. Since $r:WH(\B/\A)\to S(\A)$ is a $*$-epimorphism, we have that $I\triangleleft_m S(\A)$. Choose $u\in J(I)$ and then $I'\subseteq I^*\cap WH_u(\B)$ with $I'\triangleleft_m WH_u(\B)$. Let $u'\in J(I'),\ I^+=cl(Hu')$ and $I_0^+=I^+\cap WH(\B/\A)$. By Lemma \ref{3.0.1} and the minimality of $I^*$ we have that $I_0^+=WH(\B/\A)u'=I^*$.
\end{proof}
\begin{lemma}\label{3.0.3} Assume $I_i^+\triangleleft_m WH(\B/\A)$ and $\R_i=\{\Img d_{q'}:q'\in I_i^+\},i=0,1$. \\
(1) If $R_0,R_1\in\R_i$ and $R_0\subseteq R_1$, then $R_0=R_1$.\\
(2) $\R_0=\R_1$.
\end{lemma}
\begin{proof}. (1) Let $K$ be the common kernel of $d_x,x\in I^+_i$. By Fact \ref{1.4}, the algebras $R_0,R_1$ are sections of the family of $K$-cosets in $\B$, so $R_0\subseteq R_1$ implies $R_0=R_1$.

(2) Chooose $v_i\in J(I_i^+),i=0,1$. Let $x\in I_0^+$. Hence $xv_1\in I_1^+$ and $xv_1v_0\in I_0^+$. Therefore $\Img d_x,\Img d_{xv_1v_0}\in\R_0$ and $\Img d_{xv_1}\in\R_1$. We have that
$$\Img d_{xv_1v_0}\subseteq\Img d_{xv_1}\subseteq \Img d_x$$
hence $\Img d_{xv_1v_0}\subseteq\Img d_x$ and by (1) $\Img d_{xv_1v_0}=\Img d_x$. Therefore also $\Img d_x=\Img d_{xv_1}\in \R_1$. So $\R_0\subseteq \R_1$. By symmetry, $\R_0=\R_1$.
\end{proof}
\begin{lemma}\label{3.0.4} Assume for $i=0,1,\ I_i^+\triangleleft_m WH(\B/\A),\ v_i\in J(I_i^+)$ and let $\Hh_i=\Hh_{v_i}=v_iI_i^+$, $I_i=r[I_i^+]\triangleleft_m S(\A)$ and $u_i=r(v_i)\in J(I_i)$. Then there are group isomorphisms $f:\Hh_0\to \Hh_1$ and $\delta:u_0I_0\to u_1I_1$ such that the following diagram commutes

\begin{center}\begin{tikzcd} \Hh_0 \arrow{r}{f} \arrow{d}{r} & \Hh_1 \arrow{d}{r}\\
u_0I_0 \arrow{r}{\delta} & u_1I_1
\end{tikzcd}
\end{center}
\end{lemma}
\begin{proof}.
Let $K_i=\Ker d_{v_i}$ and $R_i=\Img d_{v_i}, i=0,1$. First we consider two special cases.

{\em Case 1.} $K_0=K_1$. In this case by Fact \ref{1.4} and Lemma \ref{3.0.1}(2) we have that $cl(Hv_0)=cl(Hv_1)$ and
$$I_0^+=cl(Hv_0)\cap WH(\B/\A)=cl(Hv_1)\cap WH(\B/\A)=I_1^+$$
Hence also $I_0=r[I_0^+]=r[I_1^+]=I_1$.
By Fact \ref{1.4}(2) and Lemma \ref{3.0.1}(5), the groups $\Hh_i$ are isomorphic to some subgroups of $\Aut(R_i)\cong\Aut(\B/K_i)$. Hence, considering the elements of $\Hh_i$ as automorphisms of $\B/K_i$ (via the function $d$) we see that the functions $f:\Hh_0\to \Hh_1$ and $\delta:u_0I_0\to u_1I_1$ defined by  $f(x)=v_1x$ and $\delta(x)=u_1x$ are group isomorphisms and the diagram from the lemma commutes.

{\em Case 2.} $R_0=R_1$. In this case we have that $v_0v_1=v_1$ and $v_1v_0=v_0$. Therefore by Corollary \ref{2.11}, $u_0u_1=u_1$ and $u_1u_0=u_0$, which gives that also $\Img d_{u_0}=\Img d_{u_1}$. Here the groups $\Hh_i$ are isomorphic to some subgroups of $\Aut(R_0)=Aut(R_1)$, the functions $f:\Hh_0\to \Hh_1$ and $\delta:u_0I_0\to u_1I_1$ defined by $f(x)=xv_1$ and $\delta(x)=xu_1$ are group isomorphisms and the diagram from the lemma commutes.

In general, by Lemma \ref{3.0.3} there is $v_2\in J(I_0^+)$ with $\Img d_{v_2}=R_1$ and $\Ker d_{v_2}=K_0$. Let $\Hh_2=\Hh_{v_2}=v_2I_0^+$ and $u_2=r(v_2)\in J(I_0)$. By Cases 1 and 2 we get group isomorphisms $f_0:\Hh_0\to\Hh_2$, $\delta_0:u_0I_0\to u_2I_0$, $f_1:\Hh_2\to\Hh_1$ and $\delta_1: u_2I_0\to u_1I_1$ such that the corresponding diagrams commute. Then the functions $f=f_1\circ f_0$ and $\delta=\delta_1\circ \delta_0$ satisfy our demands.
\end{proof}

The next proposition summarizes the previous four lemmas.
\begin{proposition}\label{3.0.5} There are minimal left ideals in $WH(\B/\A)$. Assume $I^*\triangleleft_m WH(\B/\A)$. Then the following hold.\\
(1) The set $J(I^*)$ is nonempty.\\
(2) $I^*$ is a disjoint union of groups $u'I^*,u'\in J(I^*)$.\\
(3) The groups of the form $u'I^*$, where $I^*\triangleleft_m WH(\B/\A)$ and $u'\in J(I^*)$, are isomorphic.\\
(4) Assume $u'\in J(I^*)$. Then the restriction function $r:u'I^*\to uI$ is a group epimorphism, where $I=r[I^*]\triangleleft_m S(\A)$ and $u=r(u')\in J(I)$. Also $\Ker (r) = I'\triangleleft_m WH(\B/\A)$, where $I'=I^*\cap WH_u(\B)$.\\
(5) The minimal ideals $I^*\triangleleft_m WH(\B/\A)$ are relatively closed in $WH(\B/\A)$ and determined by the common kernel of $d_x,x\in I^*$.
\end{proposition}
\begin{proof}. (1),(2),(4) and (5) follow from Lemmas \ref{3.0.1} and \ref{3.0.2}. (3) follows from Lemma \ref{3.0.4}.
\end{proof}
We see that the group $\Hh=\Hh_{u'}$ from the beginning of this section is of the form $u'I^*$, where $I^*\triangleleft_m WH(\B/\A)$ and $u'\in J(I^*)$. We call such groups the Ellis subgroups of $WH(\B/\A)$. All of them are isomorphic.

 In this paper we show an algebraic connection between the Ellis groups of $S(\A)$ and the Ellis groups of $S(\B)$ under an assumption dual to the genericity assumption considered in \cite{ne2}. By Fact \ref{1.4} there is a bijection between the Ellis groups of $S(\B)$ and the set $\K\times\R$ (defined as in Fact \ref{1.4}). Every Ellis subgroup $u''I''$ of $S(\B)$ is determined uniquely by the pair $(K(u''I''), R(u''I''))\in\K\times\R$, where $K(u''I'')$ is the same for all Ellis subgroups $u''I''$ of $I''$, and for every $I''\triangleleft_mS(\B)$, all $R\in\R$ appear as $R(u''I'')$ for $u''\in J(I'')$. Hence the genericity assumption that in $S(\B)$ there is exactly one minimal left ideal $I''$ is equivalent to $|\K|=1$.
Here we consider a dual assumption that $|\R|=1$ which means that every (equivalently some) minimal left ideal of $S(\B)$ is a group \footnote{In fact, this assumption is equivalent to distality of every (equivalently: some) minimal left ideal in $S(\B)$, regarded as an $H$-flow. This will be proved in a forthcoming paper.}.
\begin{theorem}\label{3.1} Assume every minimal left ideal in $S(\B)$ is a group. Then the Ellis groups of $S(\A)$ are isomorphic to some closed subgroups of the Ellis groups of $S(\B)$.
\end{theorem}
Before the proof we shall prove two lemmas.
\begin{lemma}\label{3.2} Assume $I''\triangleleft_mS(\B)$ is a group and $u''$ is the only idempotent in $I''$. Then the function $\varphi:S(\B)\to I''$ mapping $x$ to $x*u''$ is a $*$-epimorphism and a continuous retraction.
\end{lemma}
\begin{proof}. Let $x,y\in S(\B)$. Since $yu''\in I''$ and $u''$ is the identity element of $I''$, we have that $yu''=u''yu''$. Therefore
$$\varphi(xy)=xyu''=xu''yu''=\varphi(x)\varphi(y)$$
hence $\varphi$ is a $*$-homomorphism. The rest is easy.
\end{proof}
\begin{lemma}\label{3.3} Assume every minimal left ideal in $S(\B)$ is a group. Then every minimal left ideal in $S(\A)$ is  a group.
\end{lemma}
\begin{proof}. Let $I\triangleleft_m S(\A)$ and by Lemma \ref{2.13}(2) choose $I''\triangleleft_m S(\B)$ with $r[I'']=I$. Let $u''\in J(I'')$ and then let $u\in J(I)$ be such that $r(u'')\in uI$. By Corollary \ref{2.11} $r:CH(\B/\A)\to S(\A)$ is a $*$-isomorphism.
Let $I'$ be a minimal left ideal in $CH(\B/\A)$ with $r[I']=I$ and choose $u'\in J(I')$ with $r(u')=u$.

Suppose for contradiction that $I$ is not a group. So there is $u_0\in J(I)$ distinct from $u$. Choose $u_0'\in J(I')$ with $r(u_0')=u_0$. We have that both $u'u''$ and $u''$ belong to  $I''$ that is a group, hence there is a $p\in I''$ with $u'u''p=u''$. Therefore
$$u'u''=u'(u'u''p)=(u'u')u''p=u'u''p=u''$$
and similarly $u_0'u''=u''$.

By Proposition \ref{2.9}
$$r(u'u'')=r(u')r(u'')=ur(u'')\in uI \mbox{ and } r(u_0'u'')=r(u_0')r(u'')=u_0r(u'')\in u_0I$$
Since $u\neq u_0\in J(I)$, $uI$ and $u_0I$ are disjoint groups and $ur(u'')\neq u_0r(u'')$. Hence $u'u''\neq u_0'u''$, a contradiction.
\end{proof}
\begin{proof}{\em of Theorem \ref{3.1}.} Let $I\triangleleft_m S(\A)$, $u\in J(I)$ and $\G$ be as in the beginning of this section. Due to the assumptions of Theorem \ref{3.1} and by Lemma \ref{3.3} we have that $\G\subseteq CH(\B/\A)$ is a closed subgroup of $S(\B)$ and $r:\G\to I$ is a continuous group isomorphism. By Lemma \ref{2.13}(2) choose $I''\triangleleft_m S(\B)$ with $r[I'']=I$ and pick $u''\in J(I'')$. By Lemma \ref{3.2} the function $\varphi:\G\to I''$ mapping $x$ to $x*u''$ is a continuous group homomorphism. Let $\G'=\varphi[\G]$. So $\G'$ is a closed subgroup of $I''$ and $\varphi:\G\to\G'$ is a group epimorphism. Let $p=r(u'')$ and $p'$ be the group inverse of $p$ in the group $I$. Let $\chi:I\to I$ be the right translation by $p'$. $\G\subseteq WCH(\B/\A)$, so by Proposition \ref{2.9}, for $x\in\G$ we have that
$$r(\varphi(x))=r(xu'')=r(x)r(u'')=r(x)p$$
Therefore the following diagram commutes.
\begin{center}\begin{tikzcd}
\G \arrow{r}{\varphi} \arrow{d}{r} & \G' \arrow{d}{r}\\
I  & I \arrow{l}{\chi}
\end{tikzcd}
\end{center}
Since in this diagram $r:\G\to I$ is a bijection and $r=\chi r\varphi$, also $\varphi$ is a bijection, hence a group isomorphism. Let $f:\G'\to I$ be the composition $\chi r$. So also $f=r\varphi^{-1}$. We see that $f$ is a continuous group isomorphism.
\end{proof}
Notice that the function $r:\G'\to I$ in the proof of Theorem \ref{3.1} is a twisted isomorphism, i.e. $r$ is the composition of isomorphism $f$ with right translation by $p$. The group $\G'$ in the proof of Theorem \ref{3.1} depends on the choice of $I''\triangleleft_mS(\B)$. However it is independent of this choice up to isomorphism.

In the model theoretic set-up Theorem \ref{3.1} translates into the following.
\begin{theorem}\label{3.4} Assume every minimal left ideal in $S_{ext,G}(N)$ is a group. Then the Ellis groups of $S_{ext,G}(M)$ are isomorphic to some closed subgroups of the Ellis groups of $S_{ext,G}(N)$.
\end{theorem}

Under the dual genericity assumption we have the following theorem.
\begin{theorem}\label{3.5} Assume $S(\B)$ has a unique minimal left ideal. Then the Ellis groups of $S(\A)$ are homomorphic images of subgroups of the Ellis groups of $S(\B)$.
\end{theorem}
This theorem has been proved in \cite[Theorem 4.1]{ne2} in the model theoretic set-up and the proof translates directly into the combinatorial set-up. However, using the group $\Hh$ described in the beginning of this section we shall give an easier proof, in a way dual to the proof of Theorem \ref{3.1}. The only properties of the group $\Hh$ we shall use will be that $\Hh\subseteq WH(\B/\A)$ and $r:\Hh\to uI$ is a group epimorphism. We shall need the following lemma.
\begin{lemma}\label{3.6} Assume in $S(\B)$ there is a unique minimal left ideal $I''$ and let $u''\in J(I'')$. Then the function $\psi:S(\B)\to S(\B)$ mapping $x$ to $u''*x$ is a $*$-homomorphism with the image $u''I''$.
\end{lemma}
\begin{proof}. Let $x,y\in S(\B)$. By Fact \ref{1.1}(4), $u''x$ is almost periodic, hence it belongs to the only minimal left ideal $I''$ in $S(\B)$. Also
$u''x=u''(u''x)\in u''I$.
If $x\in u''I''$, then $u''x=x$, hence the group $u''I''$ is the image of $\psi$.

Since $u''I''$ is a group and $u''x\in u''I''$, we get that $u''x=u''xu''$. Therefore
$$\psi(xy)=u''xy=u''xu''y=\psi(x)\psi(y)$$
and $\psi$ is a homomorphism.
\end{proof}
\begin{proof}{\em of Theorem \ref{3.5}.} Let $I\triangleleft_mS(\A),\ u\in J(I)$ and $\Hh\subseteq WH(\B/\A)$ be as in the beginning of this section. So $\Hh$ is a group and $r:\Hh\to uI$ is a group epimorphism. Let $I''\triangleleft_m S(\B)$ and $u''\in J(I'')$. Since $I''$ is the only minimal left ideal in $S(\B)$, by Lemma \ref{2.13}(2) we have that $r[I'']=I$. Replacing $u$ by another element of $J(I)$ if necessary (and modifying $\Hh$ accordingly), we may assume that $r(u'')\in uI$. Let $\psi:\Hh\to u''I''$ be as in Lemma \ref{3.6} and let $\Hh'=\psi[\Hh]$. So $\psi:\Hh\to\Hh'$ is a group epimorphism. Let $p=r(u'')$ and let $p'$ be the group inverse of $p$ in $uI$. Let $\chi:uI\to uI$ be the left translation by $p'$. $\Hh\subseteq WH(\B/\A)$, so by Proposition \ref{2.9} we have that
$$r(\psi(x))=r(u''x)=r(u'')r(x)=pr(x)$$
Therefore the following diagram commutes.
\begin{center}\begin{tikzcd}
\Hh \arrow{r}{\psi} \arrow{d}{r} & \Hh' \arrow{d}{r}\\
uI & uI \arrow{l}{\chi}\end{tikzcd}
\end{center}
Since in this diagram $r:\Hh\to uI$ is ``onto", $r=\chi r\psi$ and $\chi$ is a bijection, we get that $r:\Hh'\to uI$ is ``onto" and $f:\Hh'\to uI$ defined as $\chi r$ is a group epimorphism.
\end{proof}
Notice that the function $r:\Hh'\to uI$ is a twisted group epimorphism, namely it is the epimorphism $f$ composed with left translation by $p$, similarly as in the proof of Theorem \ref{3.1} and also as in \cite{ne3}, where under another assumption it is proved that the Ellis groups in $S_{ext,G}(M)$ and $S_{ext,G}(N)$ are isomorphic and in the proof restrictions are twisted isomorphisms. However in the model theoretic set-up no example is known where the twist is really needed. So, given $M\prec^* N,\ I\triangleleft_mS_{ext,G}(M)$ and $I''\triangleleft_m S_{ext,G}(N)$ with $r[I'']=I$, is it true that there are $u\in J(I)$ and $u''\in J(I'')$ with $u=r(u'')$?

Also, the group $\Hh'$ in the proof of Theorem \ref{3.5} is independent on the choice of
$u''$, up to isomorphism.

We conclude this section considering some special cases.

\begin{remark}\label{3.7}
Assume $S(\B)$ contains a unique minimal left ideal and $CH(\B/\A)\subseteq WH(\B/\A)$.
Let $f:\Hh'\to uI$ be the group epimorphism from the proof of Theorem \ref{3.5}. Then $f$ splits, that is there is a subgroup $\Hh''$ of $\Hh'$ such that $f|_{\Hh''}:\Hh''\to uI$ is an isomorphism and $\Hh'=\Ker f\rtimes \Hh''$.
\end{remark}
\begin{proof}. We consider the situation from the proof of Theorem \ref{3.5}. By Lemma \ref{3.6} we have a $*$-homomorphism $\psi:S(\B)\to u''I''$ (mapping $x$ to $u''x$) and $\Hh'=\psi[\Hh]$.

Now we have additionally a group $\G=\{p^{\B}:p\in uI\}\subseteq CH(\B/\A)\subseteq WH(\B/\A)$ such that $r:\G\to uI$ is an isomorphism. Let $\Hh''=\psi[\G]$. So $\Hh''$ is a subgroup of $u''I''$ and as in the proof of Theorem \ref{3.1} we get that $r:\Hh''\to uI$ is a bijection. We claim that $\Hh''\subseteq\Hh'$.

Indeed, let $x\in\G$. We have that $u''$ is the identity element of the group $\Hh'$ and let $u'$ be the identity element of the group $\Hh$ (as in the proof of Theorem \ref{3.5}).  Since $\psi(u')=u''$, we have that $\psi(x)=\psi(u'xu')$. By Lemma \ref{3.0.1}, $\Hh_{u'}=u'I_0^+$, where $I_0^+=cl(Hu')\cap WH(\B/\A)\triangleleft_m WH(\B/\A)$. Hence $u'xu'\in\Hh_{u'}$ and $\psi(x)\in\Hh'$.

Therefore $f|_{\Hh''}$ is an isomorphism $\Hh''\to uI$ and we are done.
\end{proof}

\begin{proposition}\label{3.8} Assume the semigroup $S(\B)$ is commutative \footnote{In the next section we shall see model theoretic examples where this is the case.}.\\
(1) The assumptions of both Theorems \ref{3.1} and \ref{3.5} are satisfied.\\
(2) The semigroup $S(\A)$ is also commutative.\\
(3) $WCH(\B/\A)=WH(\B/\A)$\\
In particular, $CH(\B/\A)\subseteq WH(\B/\A)$ and the assumptions of Remark \ref{3.7} hold.
Also the group $\G'$ from the proof of Theorem \ref{3.1} is contained in the group $\Hh'$ from the proof of Theorem \ref{3.5} and $\Hh'=\Ker f\rtimes \G'$, where $f:\Hh'\to uI$ is the group epimorphism from the proof of Theorem \ref{3.5}.
\end{proposition}

\begin{proof}. (1) Assume $I_0,I_1\triangleleft_mS(\B)$ and $u_0,u_1\in J(I_0)$. Then $u_0=u_0u_1=u_1u_0=u_1$, hence $I_0$ is a group. Also $I_0=I_1I_0=I_0I_1=I_1$, hence $S(\B)$ contains a unique minimal left ideal.

(2) By Corollary \ref{2.11} $S(\A)\cong CH(\B/\A)\subseteq S(\B)$.

(3) Assume $q,s\in S(\B)$. Using commutativity of $S(\B)$ and $S(\A)$ we see that
$$r(q*s)=r(q)*r(s)\iff r(s*q)=r(s)*r(q)$$
hence we are done by Proposition \ref{2.9}.

For the last clause we work in the situations of the proofs of Theorems \ref{3.1} and \ref{3.5}. So we have a unique $I''\triangleleft_m S(\B)$ that is a group with identity element $u''\in J(I'')$. By Lemmas \ref{3.4} and \ref{3.6} we have $*$-homomorphisms $\varphi,\psi:S(\B)\to I''$ mapping $x$ to $xu''$ and $u''x$ respectively. Since $I''$ is a group we have that
$$xu''=u''xu''=u''x$$
hence $\varphi=\psi$.
Now $\G'=\varphi[\G]=\psi[\G]$ equals $\Hh''$ from the proof of Remark \ref{3.7}, hence it is contained in $\Hh'$. The rest follows from Remark \ref{3.7}.
\end{proof}

\section{The model theoretic set-up}
In this section we elaborate on the results from Sections 2 and 3 in the model theoretic set-up. So again we assume $M$ is an $L$-structure, $T=Th(M)$, $G=G(M)$ is an infinite $0$-definable group in $M$, $M\prec^*N$ and $H=G(N)$. We are interested in the $G$-flow $S_{ext,G}(M)$ and the $H$-flow $S_{ext,G}(N)$. This situation fits into the combinatorial set-up, with $\A=\Def_{ext,G}(M)$ and $\B=\Def_{ext,G}(N)$. For $A\in \Def_{ext,G}(M)$ we define $A ^{\#}\in \Def_{ext,G}(N)$ as in Section 1. We apply the definitions of weak heir and weak coheir here. When $q\in S_{ext,G}(N)$ is a weak [co]heir of $p\in S_{ext,G}(M)$, we say that $q$ is a weak [co]heir over $M$. Let
$$CH_G(N/M)=\{q\in S_{ext,G}(N): q\mbox{ is a coheir over }M\}$$
and let $WCH_G(N/M)$ and $WH_G(N/M)$ be the corresponding sets of weak coheirs and weak heirs, respectively.

Assume $G_1$ is a $0$-definable subgroup of $G$ and let $H_1=G_1(N)$. Then we have also a $G_1$-flow $S_{ext,G_1}(M)$ and $H_1$-flow $S_{ext,G_1}(N)$, they are naturally identified with the clopen sets $S_{ext,G}(M)\cap[G_1]$ and $S_{ext,G}(N)\cap [H_1]$, respectively. So $S_{ext,G_1}(M)$ and $S_{ext,G_1}(N)$  are semigroups isomorphic to $\End(\Def_{ext,G_1}(M))$ and $\End(\Def_{ext,G_1}(N))$, and as such also closed sub-semigroups of $S_{ext,G}(M)$ and $S_{ext,G}(N)$, respectively.
\begin{lemma}\label{4.1} Assume $[G:G_1]$ is finite and $p\in S_{ext,G_1}(N)$.\\
(1) $p\in WH_{G_1}(N/M)$ iff $p\in WH_G(N/M)$.\\
(2) $p\in WCH_{G_1}(N/M)$ iff $p\in WCH_G(N/M)$.
\end{lemma}
\begin{proof}. (1)$\Leftarrow$ is immediate by Proposition \ref{2.9}. For $\Rightarrow$, let $b_1,\dots,b_n$ be representatives of the left cosets of $G_1$ in $G$. Let $s\in S_{ext,G}(N)$. Then there are $s'\in S_{ext,G_1}(N)$ and $i\leq n$ such that $\hat{b}_i*s'=s$. Hence, using the fact that $\hat{b}_i\in WH_G(N/M)\cap WCH_G(N/M)$ (Lemma \ref{2.12}), we have that
$$r(s*p)=r(\hat{b}_i*s'*p)=r(\hat{b}_i)*r(s'*p)=r(\hat{b}_i)*r(s')*r(p)=r(\hat{b}_i*s')*r(p)=r(s)*r(p)$$
So by Proposition \ref{2.9} $p\in WH_G(N/M)$. (2) has a similar proof.
\end{proof}

Now we shall characterize the notions of weak heirs and weak coheirs in the case where $T$ is stable. So from now on until Lemma \ref{4.8} we assume $T$ is stable and for simplicity also that $G=M$. This means that $\Def_{ext,G}(M)=\Def(M)$, the algebra of definable subsets of $M$, and $S_{ext,G}(M)=S(M)$, the space of complete $1$-types over $M$. Also now $M\prec^*N$ means just that $M\prec N$.

For $A\in \Def(M)$ let $\varphi_A(x;y)$ be an $L(M)$-formula saying that $x\in y\cdot A$. Likewise let $\psi_A(x;y)$ be an $L(M)$-formula saying that $x\in A\cdot y$. Let $$\Delta_M=\{\varphi_A(x;y):A\in \Def(M)\}\mbox{ and }\Delta_M^*=\{\psi_A(x;y):A\in\Def(M)\}$$ We identify formulas that are equivalent in $\M$.
\begin{remark}\label{4.2} Every $L$-formula $\varphi(x)$ over $M$ is equivalent to a $\Delta_M$-formula over $M$ and to a $\Delta_M^*$-formula over $M$.
\end{remark}
\begin{proof}. Let $A=\varphi(M)$. So $A\in\Def(M)$ and $A=\varphi_A(M,e)=\psi_A(M;e)$, where $e$ is the identity element of $G$. Hence $\varphi(x)$ is equivalent to $\varphi_A(x;e)$ and to $\psi_A(x;e)$.
\end{proof}
The following theorem characterizes weak heirs and weak coheirs in terms of local forking.
\begin{theorem}\label{4.3} Assume $q(x)\in S(N)$.\\
(1) $q\in WH_G(N/M)$ iff $q|_{\Delta_M}$ does not fork over $M$.\\
(2) $q\in WCH_G(N/M)$ iff $q|_{\Delta_M^*}$ does not fork over $M$.
\end{theorem}
Before the proof we need some introductory comments. Let $p\in S(M)$ and $\varphi(x;y)\in L(M)$. Let $d_p\varphi(y)$ be the $L(M)$-formula defining $p|_{\varphi}$, that is for every $c\in M$ we have that $\models d_p\varphi(c)$ iff $\varphi(x,c)\in p$. On the other hand we have the function $d_p:\Def(M)\to\Def(M)$, defined in Definition \ref{1.2}.
\begin{lemma}\label{4.4}
Let $q(x)\in S(N),\ p=q|_M\in S(M),\ A\in \Def(M)$ and let $c\in\M$ realize $q$.\\
(1) $d_pA=(d_p\varphi_A(M))^{-1}$\\
(2) $(d_pA)^{\#}=d_qA^{\#}$ iff $d_p\varphi_A(y)\equiv d_q\varphi_A(y)$.\\
(3) $d_qA^{\#}=\psi_A(N;c^{-1})$
\end{lemma}
\begin{proof}. (1) Let $b\in M$. We have that
$$b\in d_pA\iff b^{-1}A\in p\iff \varphi_A(x;b^{-1})\in p\iff b\in (d_p\varphi_A(M))^{-1}$$

(2)  By (1), $(d_pA)^{\#}=(d_p\varphi_A(N))^{-1}$ and $d_qA^{\#}=(d_q\varphi_{A^{\#}}(N))^{-1}$. So (2) follows.

(3) Let $b\in N$. We have that
$$b\in d_qA^{\#}\iff c\in b^{-1}A^{\M}\iff b\in A^{\M}c^{-1}\iff b\in \psi_A(N;c^{-1})$$
where $A^{\M}$ is the subset of $\M$ defined by the formula defining $A$ in $M$. So we are done.
\end{proof}
\begin{proof}{\em of Theorem \ref{4.3}.} Let $p=q|_M$. (1)$\Rightarrow$ Assume that $q$ is a weak heir over $M$. Then for every $A\in\Def(M),\ (d_pA)^{\#}=d_qA^{\#}$, hence by Lemma \ref{4.4}(2), $d_q\varphi_A(y)\equiv d_p\varphi_A(y)$.

Let $p'\in S(\M)$ be the nonforking extension of $p$. Then for every formula $\chi(x;\bar{y})$, the formula $d_p\chi(\bar{y})$ defines $p'|_{\chi}$. When $\chi=\varphi_A$, we have that $d_p\chi\equiv d_q\chi$, hence $q|_{\Delta_M}\subseteq p'$ and $q|_{\Delta_M}$ does not fork over $M$.

$\Leftarrow$ Assume that $q|_{\Delta_M}$ does not fork over $M$. Let $q'\in S(\M)$ be an extension of $q|_{\Delta_M}$ that does not fork over $M$. By Remark \ref{4.2}, $q|_{\Delta_M}$ contains a complete type over $M$, hence $p\subseteq q|_{\Delta_M}$ and $q'$ is a nonforking extension of $p$. Therefore for every $A\in \Def(M)$ we have
$$d_p\varphi_A(y)\equiv d_{q'}\varphi_A(y)\equiv d_q\varphi_A(y)$$
Thus by Lemma \ref{4.4}(2), $(d_pA)^{\#}=d_qA^{\#}$ and $q$ is a weak heir of $p$.

(2)$\Rightarrow$ Assume $q$ is a weak coheir over $M$. Hence for every $s\in S(N)$ and $A,B\in\Def(M)$, if $d_sA^{\#}\cap B^{\#}\in q$, then $d_sA^{\#}\cap B^{\#}\cap M\neq\0$. By Lemma \ref{4.4}(3), $d_sA^{\#}=\psi_A(N;c^{-1})$, where $c\in\M$ realizes $s$. Hence we get that for every $c\in N$, if $\psi_A(x;c^{-1})\in q$ and $\varphi(x)\in p$, then the formula $\psi_A(x;c^{-1})\land\varphi(x)$ is realized by an element of $M$ (consider $s=tp(c/M)$ and $B=\varphi(M)$).

Let $p|N$ be the coheir of $p$ in $S(N)$. We see that $q|_{\Delta_M^*}\subseteq p|N$. Clearly $p|N$ does not fork over $M$, hence neither does $q|_{\Delta_M^*}$.

$\Leftarrow$ Assume that $q|_{\Delta_M^*}$ does not fork over $M$. The type $p=q|_M\in S(M)$ is stationary, therefore $q|_{\Delta_M^*}$ is the only nonforking extension of $p$ in $S_{\Delta_M^*}(N)$. Let $q'\in S(N)$ be a weak coheir of $p$. By (2)$\Rightarrow$, $q'|_{\Delta_M^*}$ does not fork over $M$. Hence $q|_{\Delta_M^*}=q'|_{\Delta_M^*}$. We shall prove that $q$ is also a weak coheir over $M$.

Let $s\in S(N),\ A,B\in\Def(M)$ and assume that $d_sA^{\#}\cap B^{\#}\in q$. We have that

\begin{quote}$(*)$\ \ \ \ \ \ \ \ \ \ \ \ \ \ \ \ \ \ \ \ \ \ \ $d_sA^{\#}\cap B^{\#}\in q|_{\Delta_M^*}$
\end{quote}
meaning that $d_sA^{\#}\cap B^{\#}$ is definable (in $N$) by a $\Delta_M^*$-formula (with parameters from $N$) belonging to $q$.
Indeed, first notice that obviously $B^{\#}\in q|_{\Delta_M^*}$. We shall prove that also $d_sA^{\#}\in q|_{\Delta_M^*}$.

Let $d\in\M$ realize $s$ and let $s^{-1}=tp(d^{-1}/N)$. By Lemma \ref{4.4}(3),
$$d_sA^{\#}=\psi_A(N;d^{-1})=d_{s^{-1}}\psi_A^*(N)$$
where $\psi_A^*(y;x)$ is the formula $\psi_A(x;y)$ with the roles of variables $x,y$ reversed, so that in $\psi_A^*$, $x$ is the parameter variable. By \cite[Lemma I.2.2]{pillay}, $d_{s^{-1}}\psi^*_A (x)$ is equivalent in $N$ to a positive Boolean combination of formulas $\psi_A(x,c),c\in N$, hence it is a $\Delta_M^*$-formula over $N$. This proves $(*)$.

Since $q|_{\Delta_M^*}=q'|_{\Delta_M^*}$, we get that $d_sA^{\#}\cap B^{\#}\in q'$. As $q'$ is a weak coheir over $M$, we conclude that $d_sA^{\#}\cap B^{\#}\cap M\neq\0$ and $q$ is a weak coheir over $M$.
\end{proof}
By Theorem \ref{4.3}, in the stable case both sets $WCH_G(N/M)$ and $WH_G(N/M)$ are closed in $S_G(N)$. Also, in this case the groups $\G,\G',\Hh$ and $\Hh'$ from Section 3 are equal to the set $CH_G(N/M)$ that is thhe set if generic types in $S_G(N)$. It is the only minimal left ideal in $S_G(N)$ and the group epimorphisms $f$ from the proof of Theorems \ref{3.1} and \ref{3.4} are restrictions.
\begin{corollary}\label{4.5} Assume $a\in\M$. Then $tp(a/N)\in WH_G(N/M)$ iff $tp(a^{-1}/N)\in WCH_G(N/M)$.
\end{corollary}
\begin{proof}. Let $i:\M\to\M$ be the group inversion, that is $i(x)=x^{-1}$. So $i$ maps $a$ to $a^{-1}$. Also, given a formula $\varphi_A(x;y)\in\Delta_M$ and $b\in\M$, $i$ maps $\varphi_A(\M;b)$ to the set $\psi_{A^{-1}}(\M;b^{-1})$, since for $x,y\in\M$,
$$x\in yA\iff x^{-1}\in A^{-1}y^{-1}$$
Therefore $tp(a/B)|_{\Delta_M}$ does not fork over $M$ iff $tp(a^{-1}/N)|_{\Delta_M^*}$ does not fork over $M$, so we are done by Theorem \ref{4.3}.
\end{proof}
We can also prove Corollary \ref{4.5} directly, using Proposition \ref{2.9} and the fact that in the stable case $*$ is just the independent multiplication of types \cite{ne1}, that is for $p,q\in S(N)$, $p*q=tp(ab/N)$ for any $N$-independent $a\models p$ and $b\models q$. This fact implies also that in the stable case $*$ in $S(N)$ is continuous not only in the first coordinate, but also in the second coordinate. The next corollary corresponds to Corollary \ref{2.10}.
\begin{corollary}\label{4.6} Assume $T$ is stable and $q\in S_G(N)$. Then $q\in WCH_G(N/M)$ iff $r(q*\hat{h})=r(q)*r(\hat{h})$ for every $h\in G(N)$.
\end{corollary}
Also the proof of Theorem \ref{4.3}(2)$\Rightarrow$ shows that for $q\in S(N)$ we have that $q\in WCH_G(N/M)$ iff for every $c\in N$ and $A,B\in \Def(N)$, if $A^{\#}c\cap B^{\#}\in q$, then $A^{\#}c\cap B^{\#}\cap M\neq\0$.

Now we assume additionally that $G=M$ is abelian. Since in a stable theory forking independence is symmetric and $*$ in $S(M)$ is the free multiplication of types, we get that $(S(M),*)$ is a commutative semigroup (in fact, $(S(M),*)$ is commutative iff $M$ is abelian). In the combinatorial case this provides examples where $(S(\B),*)$ is commutative.

By Proposition \ref{3.8} we get that in the stable case, if $G$ is abelian, then $WH_G(N/M)=WCH_G(N/M)$. (This follows also from Theorem \ref{4.3}, since in the abelian case $\Delta_M=\Delta_M^*$.) We can do slightly better.
\begin{corollary}\label{4.7} If $G$ is abelian-by-finite, then $WH_G(N/M)=WCH_G(N/M)$.
\end{corollary}
\begin{proof}. Assume $G$ is abelian-by-finite. So there is an abelian subgroup $G_1$ of finite index in $G$ and by stability we may assume that $G_1$ is definable. Let $H_1=G_1(N)$. By the discussion above, $WH_{G_1}(N/M)=WCH_{G_1}(N/M)$. Let $b_1,\dots,b_n\in G$ be representatives of the left cosets of $G_1$ in $G$ and let $p\in S_G(N)$. Then there are $p'\in S_{G_1}(N)$ and $i\leq n$ such that $p=\hat{b}_i*p'$. We are done by Lemmas \ref{4.1} and \ref{2.12}.
\end{proof}

Now we waive the assumption that $T$ is stable. We say that $G$ has the {\em group ideals property} (shortly gip) if every (equivalently some) minimal left ideal in $S_{ext,G}(M)$ is a group. Let $\K_M,\R_M$ be the set of kernels and images of the functions $d_p$ for almost periodic $p\in S_{ext,G}(M)$. So $G$ has the gip iff $\R_M$ is a singleton.  In the model theoretic set-up Theorem \ref{3.1} translates into Theorem \ref{3.4}. Also Lemma \ref{3.3} translates as follows.
\begin{lemma}\label{4.8} Assume $M\prec^* N$ and the group $H=G(N)$ has the gip. Then also $G$ has the gip.
\end{lemma}
Is the gip assumption model theoretic? The next proposition suggests a positive answer.
\begin{proposition}\label{4.9} Assume $G(x)$ is a formula defining groups in models of $T$ and $\kappa\geq|T|$. Then the following are equivalent.\\
(1) $G(M)$ has the gip for every model $M$ of $T$.\\
(2) $G(M)$ has the gip for every model $M$ of $T$ of power $\kappa$.
\end{proposition}
\begin{proof}. (1)$\Rightarrow$(2) is obvious. We prove (2)$\Rightarrow$(1). Suppose for contradiction that (2) holds but (1) fails and in some $N\models T$, the group $H=G(N)$ does not have the gip. There are two cases, depending on whether $\|N\|<\kappa$.

{\em Case 1.} $\|N\|<\kappa$ Let $N'$ be a $*$-elementary extension of $N$ of power $>\kappa$. Since $H$ does not have the gip, by Lemma \ref{4.8}aso $H':= G(N')$ does not have the gip, so this case reduces to the following Case 2.

{\em Case 2.} $\|N\|>\kappa$ Here we need to work harder. Since $H$ does not have the gip, the set $\R_N$ has at least two distinct elements $R_0$ and $R_1$. By Fact \ref{1.5} $R_0\not\subseteq R_1$. Choose $A\in R_0$ with $A\not\in R_1$. Again by Fact \ref{1.5} the set $A$ is strongly generic in $H$ and the $H$-subalgebra of $\Def_{ext,G}(N)$ generated by $R_1\cup\{A\}$ is not generic. This means that there are a Boolean term $\tau(x_0,\dots,x_n)$ and some $B_1,\dots,B_n\in R_1$ and $h_1,\dots,h_n\in H$ such that the set $\tau(A,h_1B_1,\dots,h_nB_n)$ is nonempty and not generic in $H$.

Let $L'$ be the language $L$ of $T$ extended by new unary predicates $P_A,P_{B_1},\dots,P_{B_n}$ that we interpret in $N$ as $A,B_1,\dots,B_n$, respectively, obtaining an $L'$-structure $N'$. Let $M'$ be an elementary substructure of $N'$ of power $\kappa$, containing $h_1,\dots,h_n$ and let $M=M'|_L$. Then the sets $A'=A\cap M,B_i'=B_i\cap M,i=1,\dots,n$, are subsets of $G$ externally definable in $M$, the set $A'$ is strongly generic in $G$ and the $G$-algebra $R'$ generated by the sets $B_1',\dots B_n'$ is generic.

Let $R_0',R_1'$ be some maximal generic $G$-subalgebras of $\Def_{ext,G}(M)$ containing $A'$ and $R'$, respectively. Since $M'\prec N'$, the set $\tau(A',h_1B_1',\dots,h_nB_n')$ is nonempty and not generic in $G$. Hence $A'\not\in R_1'$ and $R_0'\neq R_1'$. Therefore $G$ does not have the gip, a contradiction.
\end{proof}
Proposition \ref{4.9} is a counterpart of \cite[Lemma 4.5]{ne2} that deals with the dual property that $G$ has external generic types.

Given $M\prec^* N$ and $p\in S_{ext,G}(M)$ that is almost periodic, weak heir and [weak] coheir
extensions are distinguished extensions of $p$ in $S_{ext,G}(N)$. Sometimes they are almost periodic, sometimes not.

For example, let $M$ be an o-minimal expansion of the ordered field of reals. Since $M$ is Dedekind complete, $\Def_{ext}(M)=\Def(M)$ \cite{marste}, hence $M\prec^* N$ means just $M\prec N$. Let $G_1=(M,+)$. Then the only almost periodic types in $S_{G_1}(M)$ are the two types at infinity $p_{-\infty}$ and $p_{+\infty}$. For every $N\succ M$ they extend uniquely to almost periodic types $p_{-\infty}',p_{+\infty}'\in S_{ext,G_1}(N)$. The types $p_{-\infty}',p_{+\infty}'$ are weak heirs and not weak coheirs over $M$, provided $N\neq M$. In this example the groups $G_1(N)$ have the gip for all $N\equiv M$. Also, if $N$ is a proper elementary extension of $M$, then the set $WH_{G_1}(N/M)$ is not closed in $S_{ext,G_1}(N)$.

Next assume $G_2$ is a definably compact definable group in $M$. In this case by \cite{ne2} the groups $G_2(N)$ have external generic types for every $N\equiv M$ containing the parameters of the definition of $G_2$. Also, for all $N\succ M$ the almost periodic types in $S_{ext,G_2}(N)$ are exactly the coheirs of the almost periodic types in $S_{G_2}(M)$ and not weak heirs over $M$.

Combining these two examples we create a structure $M'$ with two sorts containing $G_1$ and $G_2$, respectively. Let $G=G_1\times G_2$. Then $G$ has neither external generic types nor the gip and also if $p\in S_{ext,G}(M')$ is almost periodic, $N'$ is a sufficiently large $*$-elementary extension of $M'$ and $p'\in S_{ext,G}(N')$ is an almost periodic extension of $p$, then $p'$ is neither a weak heir nor a weak coheir of $p$.

Adam Malinowski\\
 Instytut Matematyczny, Uniwersytet Wroc{\l}awski, pl.Grunwaldzki 2, 50-384 Wroc{\l}aw, Poland\\
{\em E-mail address:} aadam.malinowski at gmail.com\\

Ludomir Newelski (corresponding author)\\
 Instytut Matematyczny, Uniwersytet Wroc{\l}awski, pl.Grunwaldzki 2, 50-384 Wroc{\l}aw, Poland\\
{\em E-mail address:} Ludomir.Newelski at math.uni.wroc.pl

\end{document}